\documentclass[a4paper,11pt]{article}
\usepackage[T1]{fontenc}
\usepackage{amsfonts}
\usepackage{textcomp}
\usepackage{amsmath, amsthm, amssymb, mathrsfs}
\usepackage{graphicx, url}
\usepackage{xcolor}
\usepackage[utf8]{inputenc}
\usepackage[margin=1in]{geometry}
\usepackage{framed}
\usepackage[round]{natbib} 
\usepackage{algorithm}
\usepackage{algpseudocode}
\usepackage{hyperref}

\title{Exact Simulation from Tempered Stable Distributions with Infinite Variation ($\alpha\ge1$)}
\author{Michael Grabchak\footnote{Email address: mgrabcha@charlotte.edu}\\
{\it University of North Carolina Charlotte}}

\begin{document}

\newtheorem{claim}{Claim}
\newtheorem{prop}{Proposition}
\newtheorem{thrm}{Theorem}
\newtheorem{defn}{Definition}
\newtheorem{cor}{Corollary}
\newtheorem{lemma}{Lemma}
\newtheorem{remark}{Remark}
\newtheorem{exam}{Example}
\newtheorem{algo}{Algorithm}

\newcommand{\rd}{\mathrm d}
\newcommand{\rE}{\mathrm E}
\newcommand{\rP}{\mathrm P}
\newcommand{\Exp}{\mathrm{Exp}}
\newcommand{\iid}{\stackrel{\mathrm{iid}}{\sim}}
\newcommand{\eqd}{\stackrel{d}{=}}

\maketitle

\begin{abstract}
We develop the first exact and computationally tractable method for simulating from tempered stable distributions in the infinite variation case, which corresponds to $\alpha\in[1,2)$. A small simulation study shows that the approach works well.
\end{abstract}

\section{Introduction}

Tempered stable (TS) distributions are a class of models obtained by modifying the tails of stable distributions to make them lighter, which leads to models that are more realistic in practice. They have been used successfully in fields such as finance, physics, computer science, actuarial science, and biostatistics, see the references in \cite{Rachev:Kim:Bianchi:Fabozzi:2011} or \cite{Grabchak:2016}. TS distributions were introduced in \cite{Tweedie:1984}, rediscovered in \cite{Koponen:1995}, and then generalized and systematically studied in \cite{Rosinski:2007}. In the finite variation case (where $\alpha<1$), several exact simulation methods are known. These include the simple rejection sampling method of \cite{Brix:1999} and the double rejection method of \cite{Devroye:2009}. Other exact methods can be found in \cite{Hofert:2011}, \cite{Dassios:Qu:Zhao:2018}, \cite{Qu:Dassios:Zhao:2021}, and \cite{Grabchak:2021}. However, to date, no computationally tractable exact simulation method for the infinite variation case (where $\alpha\ge1$) has been proposed. There are several approximate methods, which include approaches based on truncating an infinite series representation \citep{Rosinski:2007} and on rejection sampling  \citep{Baeumer:Meerschaert:2010}. A comparative survey of these and other approximate methods, along with several new proposals, is provided by \cite{Kawai:Masuda:2011}.

In this paper, we develop the first computationally tractable exact method for simulating TS distributions in the infinite variation case ($\alpha\ge1$). This enables simulation of related stochastic processes, including the corresponding L\'evy and Ornstein-Uhlenbeck processes, see \cite{Grabchak:Sabino:2023}. This also allows for the use of simulation-based inference, which is important since the distribution and density functions of TS distributions are computationally difficult to evaluate. Our simulation methodology is based on rejection sampling and relies on the following observation. While the densities of TS distributions are difficult to evaluate, they can be viewed as marginal distributions of a joint distribution with a simple density. In the finite variation case, simulation methods based on such a representation are given in \cite{Devroye:2009} and \cite{Qu:Dassios:Zhao:2021}.

The rest of this paper is organized as follows. In Section \ref{sec: main}, we present our main results. In Section \ref{sec: sim alg}, we explain how to use these results for simulation and give our algorithms. In Section \ref{sec: sims}, we present the results of a small simulation study. In Section \ref{sec: ref}, we introduce a refinement of our methodology that can improve performance in some cases. In Section \ref{sec: ext}, we briefly discuss 
extensions of our methodology to the simulation of several related classes of distributions and give some directions for future work.

Before proceeding, we introduce our notation. We write $N(0,1)$ to denote the standard normal distribution, $\Exp(1)$ to denote the exponential distribution with mean $1$, and, for $a<b$, we write $U(a,b)$ to denote the uniform distribution on $(a,b)$. For $x\in\mathbb R$ we define
$$
\mathrm{sign}(x) =\begin{cases}
1 & \mbox{if}\ \ x>0\\
0 & \mbox{if}\ \ x=0\\
-1 & \mbox{if}\ \  x<0
\end{cases}.
$$

\section{Main Results}\label{sec: main}

Fix $\alpha\in(0,2)$ and let $f$ be the density of a fully left-skewed $\alpha$-stable distribution with scale parameter $\sigma>0$. Proposition 1.2.12 in \cite{Samorodnitsky:Taqqu:1994} tells us that for $z\ge0$
\begin{eqnarray}\label{eq: mgf stable}
\int_{-\infty}^\infty e^{zx} f(x)\rd x = \begin{cases}
e^{-\sigma^\alpha\frac{z^\alpha}{\cos\left(\frac{\pi\alpha}{2}\right)} } & \mbox{if } \alpha\ne1\\
e^{\sigma \frac{2}{\pi} z\log z } & \mbox{if } \alpha=1
\end{cases}.
\end{eqnarray}
We define a fully left-skewed TS distribution as one with density
\begin{eqnarray}\label{eq: g defin}
g(x) = C e^{x/\ell} f(x),
\end{eqnarray}
where $\ell>0$ is a parameter and the normalizing constant is given by
$$
C = \begin{cases}
e^{\frac{\sigma^\alpha}{\ell^{\alpha}\cos\left(\frac{\pi\alpha}{2}\right)} } & \mbox{if } \alpha\ne1\\
 e^{\sigma\frac{2}{\ell\pi}\log\ell}  & \mbox{if } \alpha=1
\end{cases}.
$$
In this case, for $z\ge -1/\ell$, the moment generating function (mgf) is given by
\begin{eqnarray}\label{eq: mgf TS}
\int_{-\infty}^\infty e^{zx} g(x)\rd x = \begin{cases}
e^{\frac{\sigma^\alpha}{\cos\left(\frac{\pi\alpha}{2}\right)}\left(\ell^{-\alpha} - (z+1/\ell)^\alpha\right) } & \mbox{if } \alpha\ne1\\
e^{\sigma \frac{2}{\pi} \left( \ell^{-1}\log(\ell) +(z+1/\ell)\log(z+1/\ell) \right) } & \mbox{if } \alpha=1
\end{cases}.
\end{eqnarray}
We denote this distribution by $\mathrm{TS}_-(\alpha,\ell,\sigma)$, where the minus sign emphasizes that this is fully left-skewed. 

When $\alpha\in(0,1)$, the support of $f$ is $(-\infty,0]$, and hence $g(x)\le C f(x)$ in this case. This leads to a simple rejection sampling algorithm, where the proposal distribution is the underlying stable. This was first published in \cite{Brix:1999}, see also Algorithm 0 in \cite{Kawai:Masuda:2011}. On the other hand, for $\alpha\in[1,2)$, the support of $f$ is the entire real line. Consequently, no such bound and, hence, no such algorithm exists. As we will see, in this case, the proposal distribution will be, essentially, a mixture of the underlying stable distribution and a distribution that is, itself, a mixture of half normal distributions.

\begin{prop}\label{prop: get sigma}
1. Let $X\sim \mathrm{TS}_-(\alpha,\ell/\sigma,1)$. If $\alpha\ne1$, then  $\sigma X\sim \mathrm{TS}_-(\alpha,\ell,\sigma)$. If $\alpha=1$, then  $\sigma X-(2/\pi)\sigma\log\sigma \sim \mathrm{TS}_-(\alpha,\ell,\sigma)$.\\
2. If $X_1\sim \mathrm{TS}_-(\alpha,\ell,\sigma_1)$ and $X_2\sim \mathrm{TS}_-(\alpha,\ell,\sigma_2)$ are independent random variables, then $X_1+X_2\sim\mathrm{TS}_-\left(\alpha,\ell,(\sigma_1^\alpha+\sigma_1^\alpha)^{1/\alpha}\right)$.
\end{prop}

The proposition follows directly from \eqref{eq: mgf TS} and thus the proof is omitted. In light of this result, we generally focus only on the case $\sigma=1$.

\begin{remark}\label{remark: right skewed}
In the above, we are working with fully {\em left}-skewed TS distributions. However, it is more common to work with fully {\em right}-skewed TS distributions. Our choice is motivated by the fact that, for our purposes, it is convenient to place the lighter tail on the positive half-line. Ultimately, this does not matter as a random variable $X$ has a fully {\em right}-skewed TS distribution if and only if $-X$ has a fully {\em left}-skewed TS distribution.
\end{remark}

\begin{remark}
There are many parameterizations of stable distributions. In the one that \cite{Nolan:2020} calls Parametrization $1$, $f$ is the density of an $\alpha$-stable distributions with parameters $\beta=-1$, $\gamma=\sigma$, and $\delta=0$. Note that the transform in \eqref{eq: mgf stable} looks like a mgf, but we do not call it that since it does not converge on an open interval containing $0$. The transform does, however, uniquely determine the distribution. On the other hand, the transform in \eqref{eq: mgf TS} is a mgf.
\end{remark}

\subsection{Results for $\alpha\in(1,2)$}

In this subsection, we focus on the case $\alpha\in(1,2)$ and $\sigma=1$. First note that, in this case,
$$
C = e^{\frac{1}{\ell^{\alpha}\cos\left(\frac{\pi\alpha}{2}\right)} } = e^{\frac{-1}{\ell^{\alpha}\left|\cos\left(\frac{\pi\alpha}{2}\right)\right| } } <1.
$$
An explicit representation of $f$ can be found, in, e.g., Theorem 3.3 of \cite{Nolan:2020}, and is given by
\begin{eqnarray*}
f(x) =   \frac{1}{\pi}\frac{\alpha}{\alpha-1} |x|^{1/(\alpha-1)} *\begin{cases}
\int_{-\theta_0}^{\pi/2} V(\theta) e^{-x^{\alpha/(\alpha-1)}V(\theta) }\rd \theta ,& \mbox{if } x>0\\
\int_{-\pi/2} ^{-\theta_0} V(\theta) e^{-|x|^{\alpha/(\alpha-1)}V(\theta) }\rd \theta, & \mbox{if } x<0
 \end{cases},
\end{eqnarray*}
where
$$
\theta_0 = \frac{\pi}{\alpha}-\frac{\pi}{2}>0
$$
and for $\theta\in[-\pi/2,\pi/2]$
$$
V(\theta) = \left(\cos(\alpha\theta_0)\right)^{1/(\alpha-1)} \left(\cos(\theta)\right)^{1/(\alpha-1)} \cos((\alpha-1)\theta+\alpha\theta_0)\left|\sin(\alpha(\theta+\theta_0))\right|^{-\alpha/(\alpha-1)}.
$$
Since $1<\alpha<2$, it follows that $0<\theta_0<\pi/2$.

It is easily seen that $f$ is a marginal density of the joint density
\begin{eqnarray*}
f^*(x,\theta) =  \frac{1}{\pi}\frac{\alpha}{\alpha-1} |x|^{1/(\alpha-1)}  *
 \begin{cases}
V(\theta) e^{-x^{\alpha/(\alpha-1)}V(\theta) }, &  x>0, -\theta_0<\theta<\pi/2\\
V(\theta) e^{-|x|^{\alpha/(\alpha-1)}V(\theta) } ,& x<0, -\pi/2<\theta< -\theta_0
 \end{cases}
\end{eqnarray*}
and that $g$ is a marginal density of the joint density
\begin{eqnarray}\label{eq: temp stable joint pdf}
g^*(x,\theta) = C e^{x/\ell} f^*(x,\theta).
\end{eqnarray}
We now state our main result.
  
\begin{thrm}\label{thrm: main}
Fix $\epsilon\in(0,1)$ and set
$$
m_\epsilon^* = (\ell\epsilon)^{-(\alpha-1)} \frac{\alpha}{|\cos(\pi\alpha/2)|} \ \  \mbox{and} \ \  m_\epsilon^\sharp = \max\{m_\epsilon^*,1\}.
$$
1. For $x\le m^\sharp_\epsilon$ and $- \frac{\pi}{2}<\theta\le \frac{\pi}{2}$, we have
$$
g^*(x,\theta) \le C_{1} f^*(x,\theta),
$$
where
$$
C_{1}  = C e^{m_\epsilon^\sharp/\ell}.
$$
2.  For $x\ge m_\epsilon^\sharp$ and $-\theta_0<\theta\le \frac{\pi}{2}$, we have
$$
g^*(x,\theta) \le C_{2} \frac{\alpha}{\pi} e^{-\xi_\epsilon(\theta)(x-m_\epsilon(\theta))^2/2}\sqrt{\frac{2\xi_\epsilon(\theta)}{\pi}} ,
$$
where
$$
m_\epsilon(\theta) = \left(\frac{\alpha-1}{\ell \epsilon V(\theta)\alpha}\right)^{\alpha-1}, \ \ \ 
\xi_\epsilon(\theta) = (V(\theta))^{\alpha-1} \frac{(\epsilon\alpha)^{\alpha-1}}{(\alpha-1)^\alpha}\ell^{-(2-\alpha)},
$$
and
$$
C_{2} = C^{1-\epsilon^{1-\alpha}}  \sqrt \pi \frac{(\epsilon\alpha)^{(1-\alpha)/2}}{(\alpha-1)^{1-\alpha/2}} \left(\frac{3-\alpha}{e(1-\epsilon)}\right)^{\frac{3-\alpha}{2}} \ell^{1-\alpha/2} 2^{-(4-\alpha)/2}.
$$
\end{thrm}

Before giving the proof, we introduce two lemmas.

\begin{lemma}\label{lemma: V prop}
1. $V$ is nonnegative, strictly increasing on $(-\pi/2,-\theta_0)$, strictly decreasing on $(-\theta_0, \pi/2)$, and takes the values: $V(-\pi/2)=0$, $V(-\theta_0)=\infty$, and
$$
V(\pi/2) = (\alpha-1)\left(\alpha^{-\alpha}|\cos(\pi\alpha/2)|\right)^{1/(\alpha-1)}.
$$
2. For $-\theta_0<\theta\le \frac{\pi}{2}$, we have $m_\epsilon(\theta)\le m^*_\epsilon\le m^\sharp_\epsilon$ and $V(\theta)\ge V(\pi/2)$.
\end{lemma}

\begin{proof}
The first part is given in Lemma 3.8 of \cite{Nolan:2020}, although we use different notation. The second follows from the first and the easily checked fact that $m_\epsilon^* = m_\epsilon(\pi/2)$.
\end{proof}

\begin{lemma}\label{lemma: Taylor}
1. Let $p>2$, $a,b>0$, and 
$
m= \left(\frac{a}{bp}\right)^{1/(p-1)}.
$
If
$$
L(x) = ax- b x^{p}, \ \ x\ge0,
$$
then for $x\ge m$
$$
L(x) \le  L(m) - \frac{|L''(m)|}{2}(x-m)^2 
$$
where
$$
|L''(m)| =  bp(p-1)m^{p-2}.
$$
2. For any $t\ge0$ and $a,b>0$
\begin{eqnarray}\label{eq: bound with exp}
t^{a} e^{-bt}\le (a/b)^{a} e^{-a}.
\end{eqnarray}
\end{lemma}

\begin{proof}
The second part can be checked using basic calculus. For the first part, we have
$$
L'(x) = a-bpx^{p-1}, \ \  L''(x) = -bp(p-1)x^{p-2} <0,
$$
and $L'''(x) <0$. Noting that $L'(m)=0$  and applying Taylor's theorem gives
$$
L(x) = L(m) - \frac{|L''(m)|}{2}(x-m)^2 -  \frac{|L'''(x^*)|}{6}(x-m)^3 \le L(m) - \frac{|L''(m)|}{2}(x-m)^2,
$$
where $x^*$ is some point that lies between $x$ and $m$.
\end{proof}

\begin{proof}[Proof of Theorem \ref{thrm: main}]
The first part is immediate. Henceforth, assume $x\ge m^\sharp_\epsilon$, $-\theta_0<\theta\le \frac{\pi}{2}$, and note that, by Lemma \ref{lemma: V prop}, $x\ge m_\epsilon(\theta)$. We have
\begin{eqnarray*}
&& e^{x/\ell} x^{\frac{1}{\alpha-1}} V(\theta) e^{-x^{\alpha/(\alpha-1)}V(\theta) } \\
&& \qquad=x^{\frac{1}{\alpha-1}}\left(V(\theta)\right)^{(3-\alpha)/2} e^{-(1-\epsilon)x^{\alpha/(\alpha-1)}V(\theta) } e^{x/\ell -\epsilon x^{\alpha/(\alpha-1)}V(\theta) }  \sqrt{\frac{\pi (\alpha-1)^\alpha \ell^{2-\alpha}} {2(\epsilon\alpha)^{\alpha-1} }} \sqrt{\frac{2\xi_\epsilon(\theta)}{\pi}}.
\end{eqnarray*}
Applying \eqref{eq: bound with exp} with $a=(3-\alpha)/2$ and $b=1-\epsilon$ gives
\begin{eqnarray*}
x^{\frac{1}{\alpha-1}}\left(V(\theta)\right)^{(3-\alpha)/2} e^{-(1-\epsilon)x^{\alpha/(\alpha-1)}V(\theta) }
&= &x^{\alpha/2-1} x^{\frac{\alpha}{\alpha-1} (\frac{3-\alpha}{2})}\left(V(\theta)\right)^{(3-\alpha)/2} e^{-(1-\epsilon)x^{\alpha/(\alpha-1)}V(\theta) } \\
&\le& x^{\alpha/2-1}\left(\frac{3-\alpha}{2(1-\epsilon)e}\right)^{\frac{3-\alpha}{2}} 
\le \left(\frac{3-\alpha}{2(1-\epsilon)e}\right)^{\frac{3-\alpha}{2}},
\end{eqnarray*}
where the last inequality follows from the fact that $x\ge1$ and $\alpha<2$. Next, applying Lemma \ref{lemma: Taylor} with $a=1/\ell$, $b=\epsilon V(\theta)$, and $p=\alpha/(\alpha-1)>2$ gives
\begin{eqnarray*}
e^{x/\ell}e^{- \epsilon V(\theta)x^{\alpha/(\alpha-1)}} &\le& e^{m_\epsilon(\theta)/\ell- \epsilon V(\theta)(m_\epsilon(\theta))^{\alpha/(\alpha-1)}} e^{-\xi_\epsilon(\theta)(x-m_\epsilon(\theta))^2/2} \\
&\le&  e^{\epsilon^{1-\alpha}\ell^{-\alpha}/|\cos\left(\frac{\pi\alpha}{2}\right)|} e^{-\xi_\epsilon(\theta)(x-m_\epsilon(\theta))^2/2} \\
&=&  C^{-\epsilon^{1-\alpha}} e^{-\xi_\epsilon(\theta)(x-m_\epsilon(\theta))^2/2}.
\end{eqnarray*}
Here the second inequality follows from the fact that
\begin{eqnarray*}
m_\epsilon(\theta)/\ell- \epsilon V(\theta)(m_\epsilon(\theta))^{\alpha/(\alpha-1)} &=& \left(V(\theta)\right)^{1-\alpha} \epsilon^{1-\alpha} (\alpha-1)^{\alpha-1}(\alpha \ell)^{-\alpha}\\
&\le&\left(V(\pi/2)\right)^{1-\alpha} \epsilon^{1-\alpha} (\alpha-1)^{\alpha-1}(\alpha \ell)^{-\alpha}\\
&=& \epsilon^{1-\alpha}\ell^{-\alpha}/\left|\cos\left(\frac{\pi\alpha}{2}\right)\right|,
\end{eqnarray*}
which, itself, follows from Lemma \ref{lemma: V prop}. Combining the above with \eqref{eq: temp stable joint pdf} gives the result.
\end{proof}

\subsection{Results for $\alpha=1$}

In this subsection we consider the case $\alpha=1$ and $\sigma=1$. The approach is similar to that used for $\alpha\in(1,2)$, but the formulas differ, so we provide a detailed derivation. In this case
$$
C = e^{\frac{2}{\ell\pi}\log\ell} = \ell^{\frac{2}{\ell\pi}}.
$$
An explicit representation of $f$ can be found, in, e.g., Theorem 3.3 of \cite{Nolan:2020}, and is given by
\begin{eqnarray*}
f(x) =   \frac{1}{2} e^{\pi x/2} \int_{-\pi/2}^{\pi/2} V(\theta) e^{-e^{\pi x/2} V(\theta) }\rd \theta ,
\end{eqnarray*}
where
$$
V(\theta) = \frac{1-\frac{2}{\pi}\theta}{\cos(\theta)}\exp\left\{(\theta-\pi/2)\tan(\theta)\right\}, \ \ \ \theta\in[-\pi/2,\pi/2].
$$
Lemma 3.8 in \cite{Nolan:2020} tells us that $V$ is nonnegative, strictly decreasing on $(-\pi/2, \pi/2)$, and takes the values:
$$
V(-\pi/2)=\infty \mbox{ and }V(\pi/2) = \frac{2}{\pi e}.
$$
Clearly, $f$ is a marginal density of the joint density
\begin{eqnarray*}
f^*(x,\theta) = \frac{1}{2} e^{\pi x/2} V(\theta) e^{-e^{\pi x/2} V(\theta)} \ \  x\in\mathbb R, -\pi/2<\theta< \pi/2.
\end{eqnarray*}
and $g$ is a marginal density of the joint density
\begin{eqnarray*}
g^*(x,\theta) = C e^{x/\ell} f^*(x,\theta).
\end{eqnarray*}
 We can now state our main result for $\alpha=1$.

\begin{thrm}\label{thrm: main a1}
Fix $\epsilon\in(0,1)$ and let
$$
m^\sharp_\epsilon = \frac{2}{\pi}\left(1-\log(\ell\epsilon)\right).
$$
1. For $x\le m^\sharp_\epsilon$ and $- \frac{\pi}{2}<\theta\le \frac{\pi}{2}$, we have 
$$
g^*(x,\theta) \le C_1 f^*(x,\theta),
$$
where
$$
C_1 = C e^{m^\sharp_\epsilon/\ell} = e^{\frac{2}{\pi\ell}\left(1-\log \epsilon\right)}.
$$
2. For $x> m^\sharp_\epsilon $ and $- \frac{\pi}{2}<\theta\le \frac{\pi}{2}$, we have
$$
g^*(x,\theta) \le C_{2} \frac{1}{\pi} e^{-\xi(x-m_\epsilon(\theta))^2/2}\sqrt{\frac{2\xi}{\pi}} ,
$$
where
$$
\xi = \frac{\pi}{2\ell}, \ \ \ 
m_\epsilon(\theta) = \frac{2}{\pi}\log\left(\frac{2}{\epsilon \ell\pi V(\theta)}\right),
$$
and
$$
C_{2} =\frac{\pi\sqrt\ell}{2e(1-\epsilon)} \epsilon^{-\frac{2}{\pi\ell}} .
$$
\end{thrm}

\begin{remark}\label{remark: alpha1 notation}
To simplify the discussion going forward, in the $\alpha=1$ case,  we sometimes set $\theta_0=\pi/2$ and write $\xi=\xi_\epsilon(\theta)$, even though $\xi$ does not actually depend on $\theta$ or $\epsilon$ in this case. This notation unifies the presentation of Theorems \ref{thrm: main} and \ref{thrm: main a1}, with the only difference being the formulas for $m^\sharp_\epsilon$, $C_1$, $C_2$, $\xi_\epsilon(\theta)$, and $m_\epsilon(\theta)$.
\end{remark}

Before proving the theorem, we give a lemma whose proof is similar to that of Lemma \ref{lemma: Taylor a1} and is thus omitted.

\begin{lemma}\label{lemma: Taylor a1}
Let $a,b,c>0$ and let
$
m= \frac{1}{b}\log\left(\frac{a}{bc}\right).
$
If
$$
L(x) = ax-  ce^{bx}, \ \ x\in\mathbb R,
$$
then, for $x\ge m$,
$$
L(x) \le L(m) - \frac{ab}{2}(x-m)^2 = \frac{a}{b}\left(\log\left(\frac{a}{bc}\right)-1\right) - \frac{ab}{2}(x-m)^2.
$$
\end{lemma}

\begin{proof}[Proof of Theorem \ref{thrm: main a1}]
The first part is immediate; we now turn to the second part. By arguments similar to those in the proof of Lemma \ref{lemma: V prop}, we have $m_\epsilon(\theta) \le m^\sharp_\epsilon$. From here, we get
\begin{eqnarray*}
C\frac{1}{2} e^{x/\ell} e^{\pi x/2} V(\theta) e^{-e^{\pi x/2} V(\theta)} &\le& C\frac{1}{2e(1-\epsilon)} e^{x/\ell-\epsilon e^{\pi x/2} V(\theta)} \\
&\le& C\frac{1}{2e(1-\epsilon)} e^{\frac{2}{\pi\ell}\left(\log\left(\frac{2}{\ell\pi\epsilon V(\theta)}\right)-1\right)} e^{-\xi (x-m_\epsilon(\theta))^2/2}\\
&\le& C\frac{1}{2e(1-\epsilon)} e^{\frac{2}{\pi\ell}\log\left(\frac{ 1}{\ell\epsilon}\right)} e^{-\xi (x-m_\epsilon(\theta))^2/2}\\
&=& \frac{\pi\sqrt\ell}{2e(1-\epsilon)} \epsilon^{-\frac{2}{\pi\ell}} e^{-\xi (x-m_\epsilon(\theta))^2/2}\sqrt\frac{2\xi}{\pi}\frac{1}{\pi},
\end{eqnarray*}
where the first inequality follows by applying \eqref{eq: bound with exp} with $a=1$ and $b=1-\epsilon$, the second by applying Lemma \ref{lemma: Taylor a1} with $a=1/\ell$, $b=\pi/2$, and $c=\epsilon V(\theta)$, and the third by the fact that $V(\theta)\ge V(\pi/2)=2/(\pi e)$.
\end{proof}

\section{Simulation Algorithms}\label{sec: sim alg}

From Theorems \ref{thrm: main} and \ref{thrm: main a1}, we immediately get the following result. See Remark \ref{remark: alpha1 notation} for an explanation of some of the notation in the $\alpha=1$ case.

\begin{cor}
For any $\epsilon\in(0,1)$ and $p_1,p_2\in(0,1)$ with $p_1+p_2=1$, we have
\begin{eqnarray}\label{eq: main bound}
g^*(x,\theta)\le K h(x,\theta),  \ \ x\in\mathbb R, \ -\frac{\pi}{2}\le\theta\le \frac{\pi}{2},
\end{eqnarray}
where
\begin{eqnarray}\label{eq: defn h}
h(x,\theta) = p_1 h_{1}(x,\theta) + p_2 h_{2}(x,\theta)
\end{eqnarray}
with
$$
h_{1}(x,\theta) = f^*(x,\theta), \ \ x\in\mathbb R,\ -\pi/2<\theta<\pi/2,
$$
$$
h_{2}(x,\theta) = \frac{\alpha}{\pi} e^{-\xi_\epsilon(\theta)(x-m_\epsilon(\theta))^2/2}\sqrt{\frac{2\xi_\epsilon(\theta)}{\pi}},  \ \ -\theta_0<\theta\le \frac{\pi}{2},\  x\in\left(m_\epsilon(\theta),\infty\right),
$$
and
$$
K = \max\left\{C_{1}/p_1, C_{2}/p_2 \right\}.
$$
\end{cor}

Here, $h_1$ and $h_2$ are bivariate densities, for which we now develop simulation methods. With these, we will be able to simulate from $h$ and, thus, through the use of rejection sampling, from $g^*$ and $g$. We can simulate from $h_1=f^*$ using Algorithm \ref{alg: h1}, which is a variant of the method of \cite{Chambers:Mallows:Stuck:1976} for simulating from stable distributions. The algorithm is justified by the following result, whose proof is similar to the proof of Theorem 1.3 in \cite{Nolan:2020} and is thus omitted. 

\begin{algorithm}[t]
    \caption{Simulation from $h_{1}=f^*$}\label{alg: h1}
    \begin{algorithmic}[1]
        \State \textbf{independently simulate} $\Theta\sim U(-\pi/2,\pi/2)$ \textbf{and} $W\sim\mathrm{Exp}(1)$
        \If{$\alpha=1$}
                        \State \textbf{set} $X = \frac{2}{\pi}\left[ \left(\frac{\pi}{2}-\Theta \right)\tan(\Theta)+ \log\left(\frac{\pi W\cos(\Theta)}{\pi-2\Theta}\right)\right]$
            \Else 
		        \State \textbf{set} $X = \mathrm{sign}(\theta_0+\Theta)\left(V(\Theta)/W\right)^{(1-\alpha)/\alpha}$
           \EndIf
       \State  \textbf{return} $( X,\Theta)$
    \end{algorithmic}
\end{algorithm}

\begin{algorithm}[t]
    \caption{Simulation from $h_{2}$. Tuning parameter $\epsilon\in(0,1)$.}\label{alg: h2}
    \begin{algorithmic}[1]
        \State \textbf{independently simulate} $\Theta\sim U(-\theta_0,\pi/2)$ and $Z\sim N\left(0, 1\right)$
        \State \textbf{set} $X = m_\epsilon(\Theta)+ |Z|/\xi_\epsilon(\Theta)$
       \State  \textbf{return} $(X,\Theta)$
    \end{algorithmic}
\end{algorithm}

\begin{prop}
Let $\Theta\sim U(-\pi/2,\pi/2)$ and $W\sim \mathrm{Exp}(1)$ be independent random variables. If $\alpha\in(1,2)$, set
$$
X = \mathrm{sign}(\theta_0+\Theta)\left(V(\Theta)/W\right)^{(1-\alpha)/\alpha}
$$
and if $\alpha=1$, set
$$
X = \frac{2}{\pi}\left[ \left(\frac{\pi}{2}-\Theta \right)\tan(\Theta)+ \log\left(\frac{\pi W\cos(\Theta)}{\pi-2\Theta}\right)\right].
$$
We have $(X,\Theta)\sim h_1=f^*$. 
\end{prop}

We can simulate from $h_2$ using Algorithm \ref{alg: h2}. To explain the algorithm, we first recall the following fact. If $Z\sim N(0,1)$ and $\sigma>0$, $\mu\in\mathbb R$, then the random variable $X=\sigma|Z|+\mu$ has density 
$$
\sqrt{\frac{2}{\pi\sigma^2}} e^{-\frac{(x-\mu)^2}{2\sigma^2}}, \ \ x>\mu.
$$
This is the half normal distribution with location parameter $\mu$ and scale parameter $\sigma$. It is easily seen that $h_2$ is the joint distribution of $\Theta$ and $X$, where the marginal distribution of $\Theta$ is $U(-\theta_0,\pi/2)$ and the conditional distribution of $X$ given $\Theta$ is half normal with location parameter $\mu =m_\epsilon(\Theta)$ and scale parameter $\sigma=1/\sqrt{\xi_\epsilon(\Theta)}$. Unconditionally, the distribution of $X$ is a mixture of half normal distributions.

\begin{algorithm}[t]
    \caption{Simulation from $h$. Tuning parameters $\epsilon\in(0,1)$, $p_1,p_2\in(0,1)$ with $p_1+p_2=1$. }\label{alg: h}
    \begin{algorithmic}[1]
        \State \textbf{simulate} $U\sim U(0,1)$ 
        \If{$U\le p_1$} 
    \State \textbf{simulate} $(X,\Theta)$ using Algorithm \ref{alg: h1} 
\Else
    \State \textbf{simulate} $(X,\Theta)$ using Algorithm \ref{alg: h2}
\EndIf 
            \State \textbf{return}  $(X,\Theta)$ 
    \end{algorithmic}
\end{algorithm}

Since $h$ has the mixture representation given in \eqref{eq: defn h}, we can simulate from it using Algorithm \ref{alg: h}. From here, the bound in \eqref{eq: main bound} allows us to simulate from $g^*$ using rejection sampling, as summarized in Algorithm \ref{alg: g}. We can now simulate from $\mathrm{TS}_-(\alpha,\ell,\sigma)$ by first simulating $(X,\Theta)$ using  Algorithm \ref{alg: g}, and then applying the result in Proposition \ref{prop: get sigma}. The procedure is summarized in Algorithm \ref{alg: TS}.

On a given iteration of Algorithm \ref{alg: g}, the probability of acceptance is $1/K$ and the expected number of iterations until the first acceptance is $K$. Thus, the smaller the value of $K$, the more efficient the algorithm. Ideally, we should select the tuning parameters ($\epsilon,p_1,p_2$) that minimize $K$. Determining the optimal way to do this is left for future work.

\begin{algorithm}[t]
    \caption{Simulation from $g^*$. Tuning parameters $\epsilon\in(0,1)$, $p_1,p_2\in(0,1)$ with $p_1+p_2=1$.} \label{alg: g}
    \begin{algorithmic}[1]
        \State \textbf{simulate} $U\sim U(0,1)$ 
        \State \textbf{simulate} $(X,\Theta)$ using Algorithm \ref{alg: h}
        \State \textbf{set} 
        $
        \varphi = \frac{g^*(X,\Theta)}{K h(X,\Theta)}
        $
       \If{$U\le \varphi$}
       \State  \textbf{return} $(X,\Theta)$
       \Else
              \State \textbf{goto} Step 1
       \EndIf
    \end{algorithmic}
\end{algorithm}

\begin{algorithm}[t]
    \caption{Simulation from $\mathrm{TS}_-(\alpha,\ell,\sigma)$, $\alpha\ge1$, $\ell>0$, $\sigma>0$.  Tuning parameters $\epsilon\in(0,1)$, $p_1,p_2\in(0,1)$ with $p_1+p_2=1$.}\label{alg: TS}
    \begin{algorithmic}[1]
    \State \textbf{simulate} $(X,\Theta)$ using Algorithm \ref{alg: g}
        \If{$\alpha=1$}
                        \State \textbf{set} $Y = \sigma X-\frac{2}{\pi}\sigma\log\sigma$
            \Else 
		        \State \textbf{set} $Y = \sigma X$
          \EndIf
       \State  \textbf{return} $Y$
    \end{algorithmic}
\end{algorithm}

\section{Simulations}\label{sec: sims}

In this section, we conduct a small simulation study to illustrate the performance of our methodology. For several choices of the parameters, we use Algorithm \ref{alg: h} to simulate $N=10^7$ observations from the proposal distribution $h$. We then apply the rejection step from Algorithm \ref{alg: g} to determine how many observations are retained. For the tuning parameters, we set  $p_1 = p_2=\frac{1}{2}$ and choose values of $\epsilon$ that seem to work well, although we did not fully optimize these values. Our results are given in Table \ref{tab: sim results} and Figure \ref{fig: sim results}. In the table, we give the values of the distribution parameters $\alpha$ and $\ell$ (we always take $\sigma=1$), our choice of the tuning parameter $\epsilon$, the value of $K$, and the number $n$ of observations retained. From the theory, we know that the algorithm works best when $K$ is small. The smallest value that we obtained is $K=2.33$, which resulted in $4,294,523$ observations, while the largest is $8.66$ and resulted in $1,155,766$ observations. In the figure, we plot the kernel density estimator (KDE) for each simulated dataset with the true density overlaid. We can see that the fit is almost perfect. To evaluate the density of the TS distribution, we used \eqref{eq: g defin}, where the stable density was evaluated using the ``stabledist'' package for R.  

\begin{table}
\centering
\begin{tabular}{ccccc}
\hline
$\alpha$ & $\ell$&  $\epsilon$ &  $K$ & $n$ \\ 
\hline
$1$ & $1$ & 0.6 & $5.23$ & 1,912,882 \\
$1$ & $2$ & 0.5 & $4.08$ & 2,455,458 \\
$1$ & $5$ & 0.1 & $3.85$ & 2,600,418 \\
$1.1$ & $1$ & 0.4 & $8.13$ & 1,230,945\\
$1.1$ & $2$ & 0.4 & $8.01$ & 1,248,302\\
$1.1$ & $5$ & 0.2 & $8.66$ & 1,155,766 \\
$1.5$ & $1$ & 0.8 & $5.21$ & 1,918,387 \\
$1.5$ & $2$ & 0.6 & $3.19$ & 3,130,498  \\
$1.5$ & $5$ & 0.3 & $3.00$ & 3,326,781 \\
$1.9$ & $1$ & 0.9 & $6.02$ & 1,659,758 \\
$1.9$ & $2$ & 0.8 & $2.86$ & 3,492,037 \\
$1.9$ & $5$ & 0.8 & $2.33$ & 4,294,523
\end{tabular}
\caption{Results of the simulation study. Here, $\alpha$ and $\ell$ are the distribution parameters, $\epsilon$ is a tuning parameter, $1/K$ is the probability of acceptance, and $n$ is the number of observations retained after simulating $N=10^7$ observations from the proposal distribution.  In all cases, we take $\sigma=1$. }\label{tab: sim results}
\end{table}

\begin{figure}
\label{fig: error in sims}
\begin{tabular}{ccc}
\includegraphics[width=0.31\textwidth]{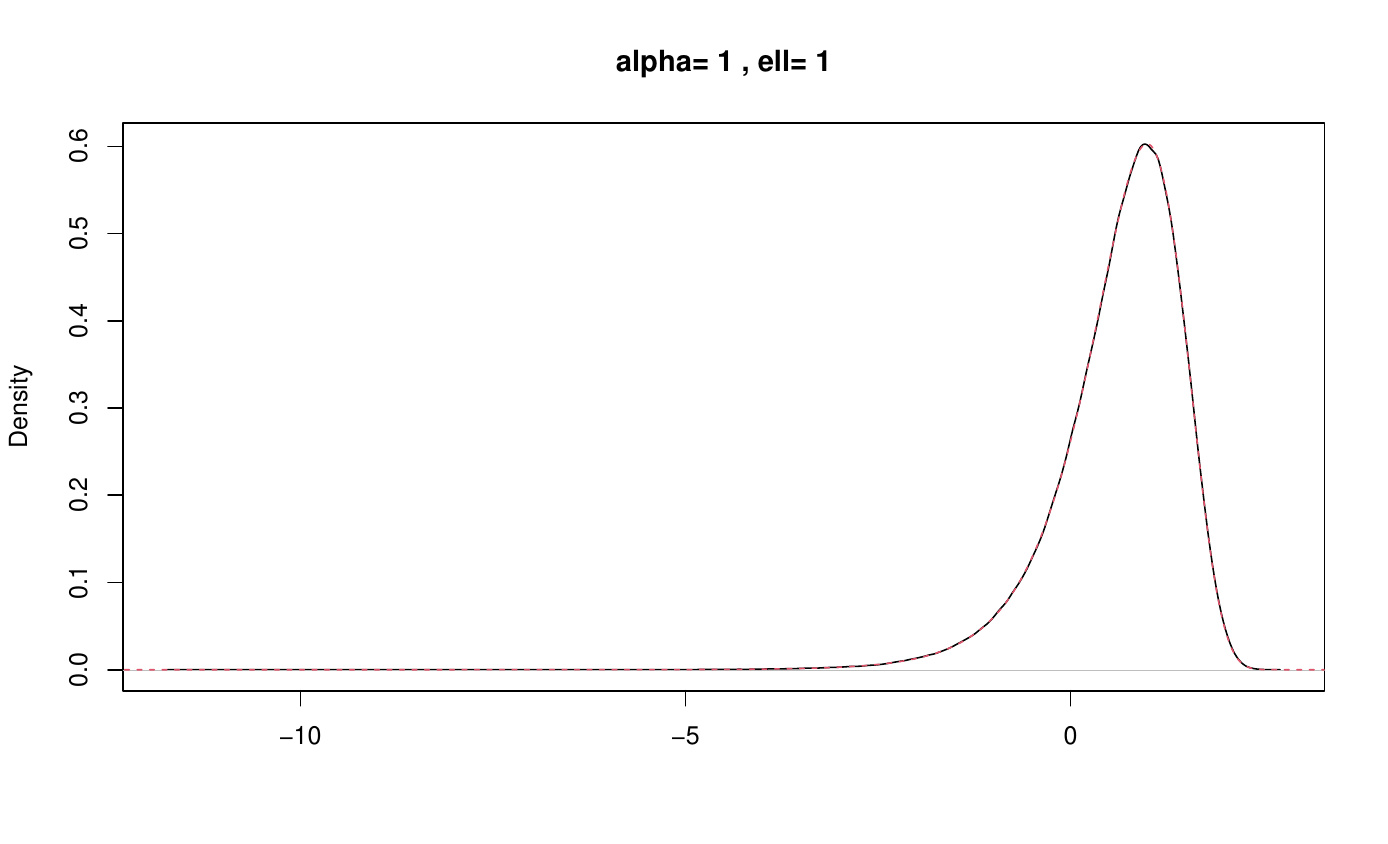} & \includegraphics[width=0.31\textwidth]{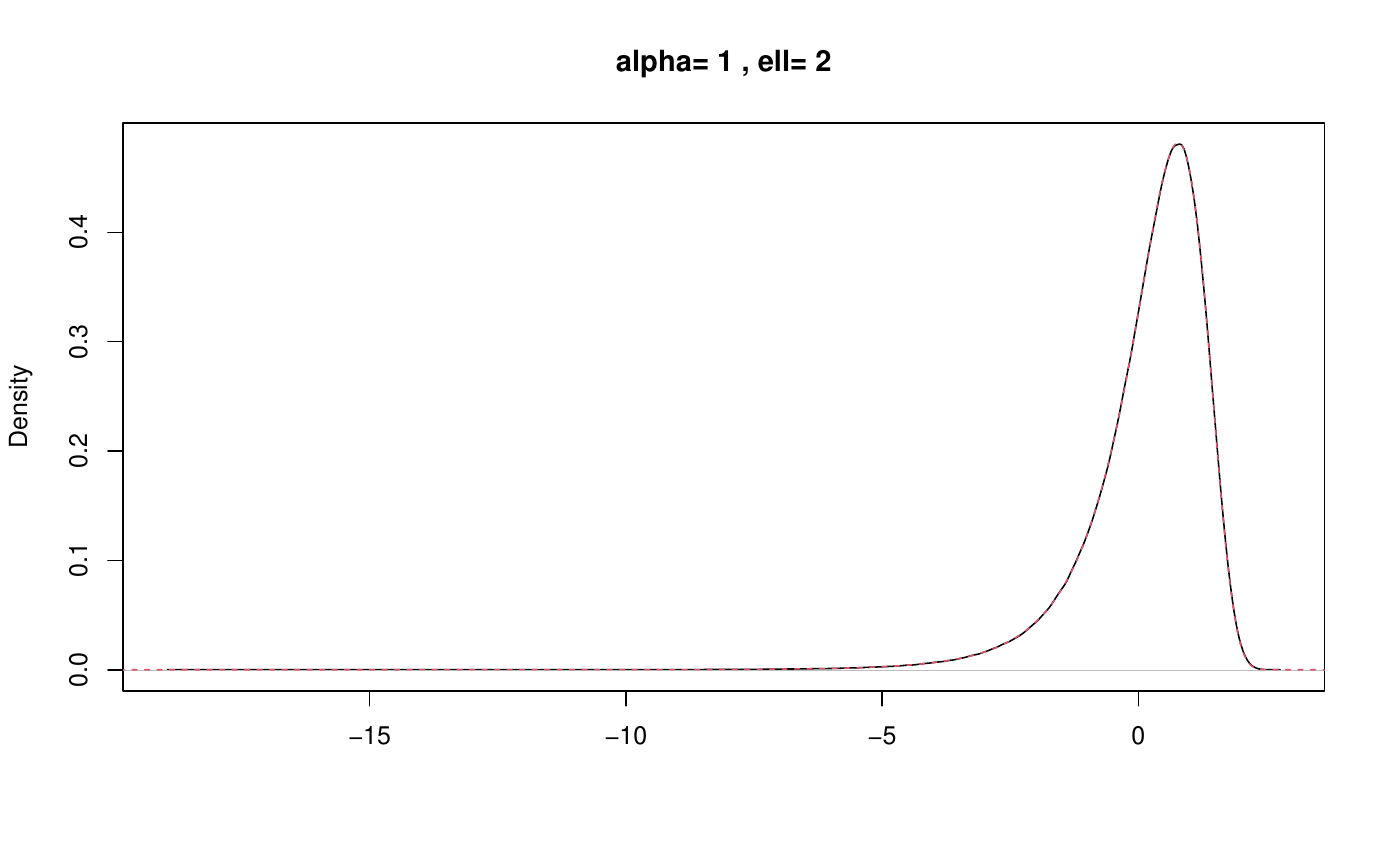} &  \includegraphics[width=0.31\textwidth]{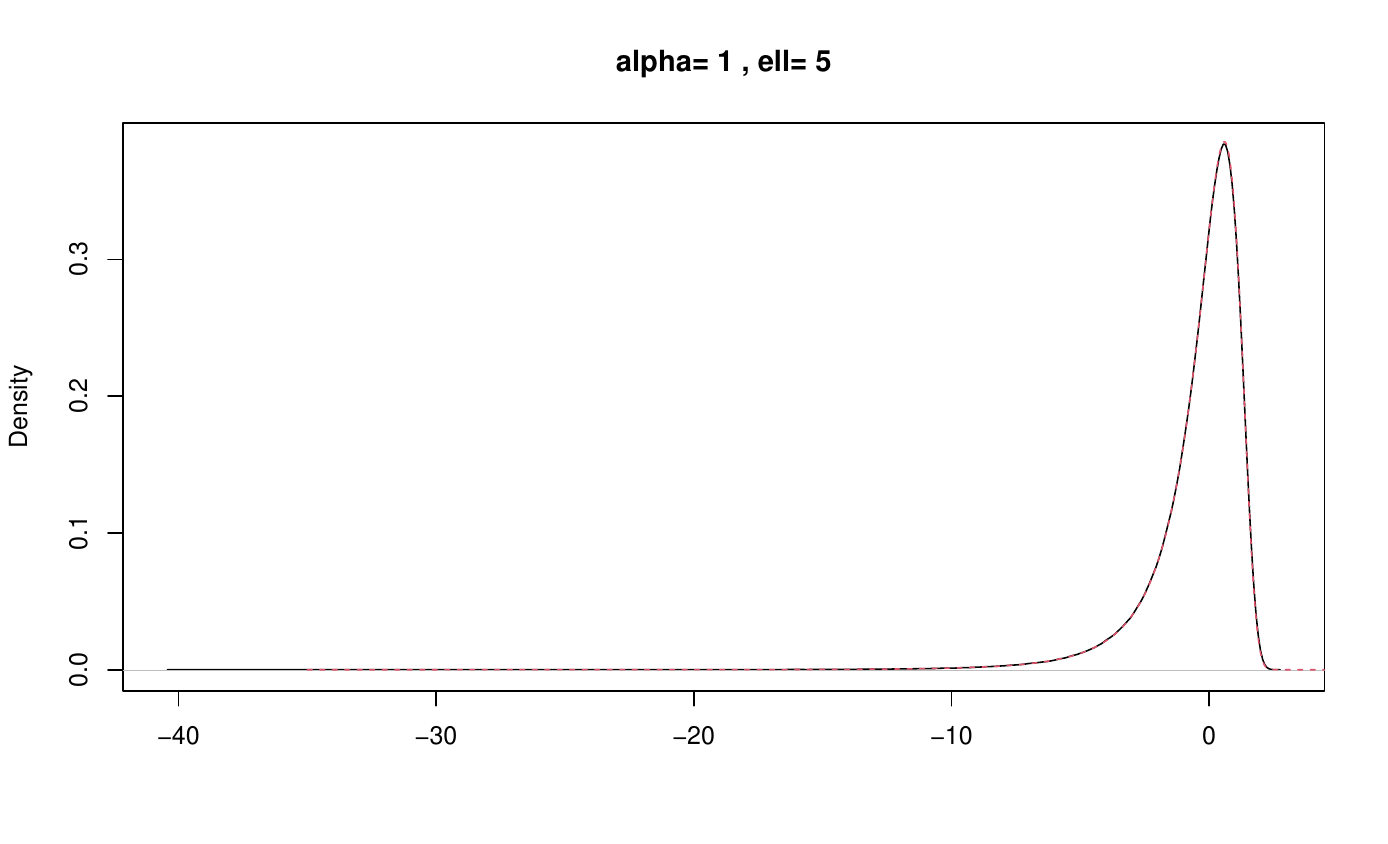}  \\
\includegraphics[width=0.31\textwidth]{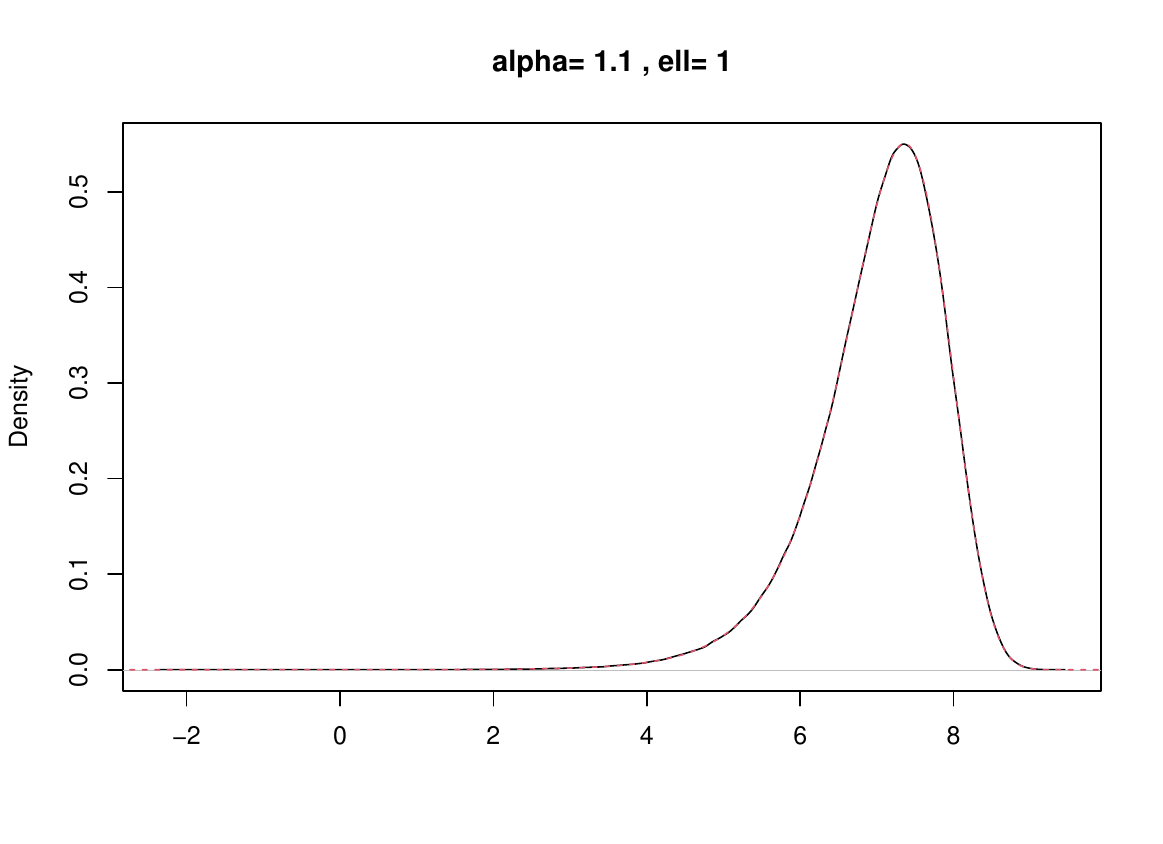} & \includegraphics[width=0.31\textwidth]{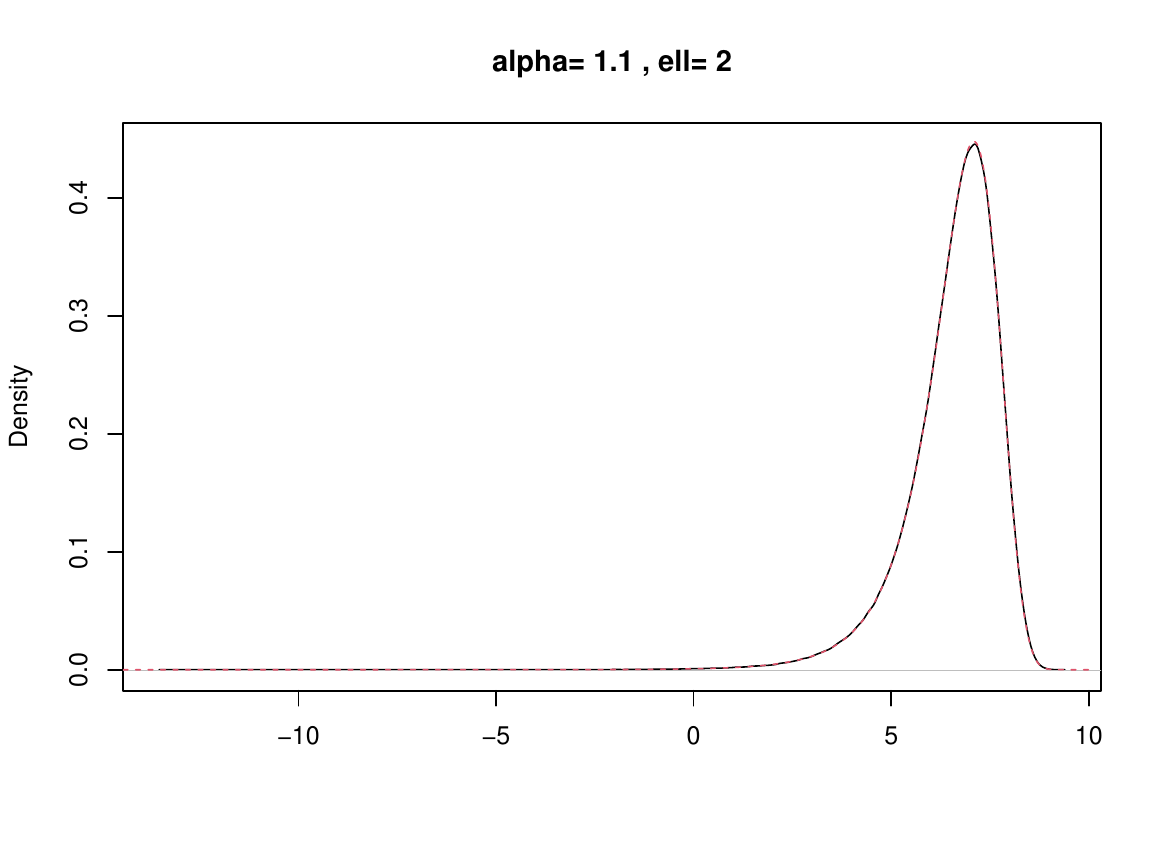} &  \includegraphics[width=0.31\textwidth]{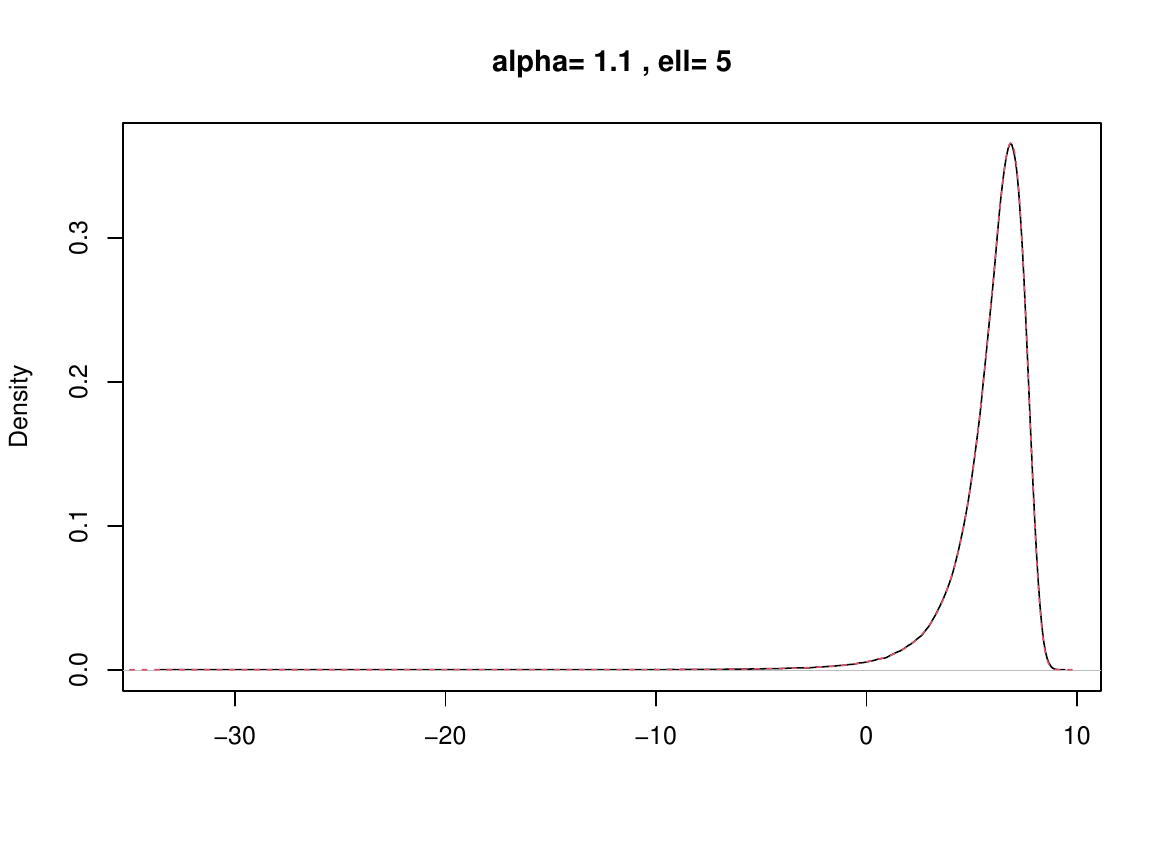}  \\
\includegraphics[width=0.31\textwidth]{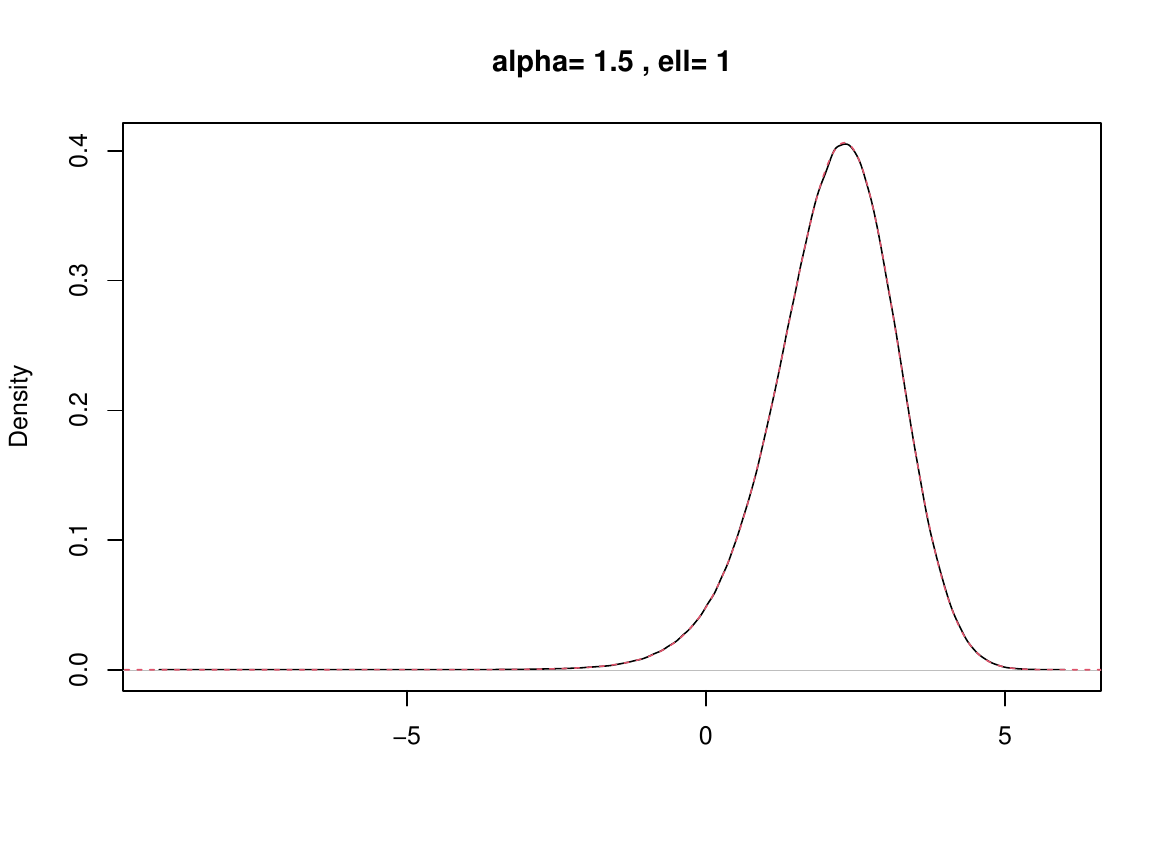} & \includegraphics[width=0.31\textwidth]{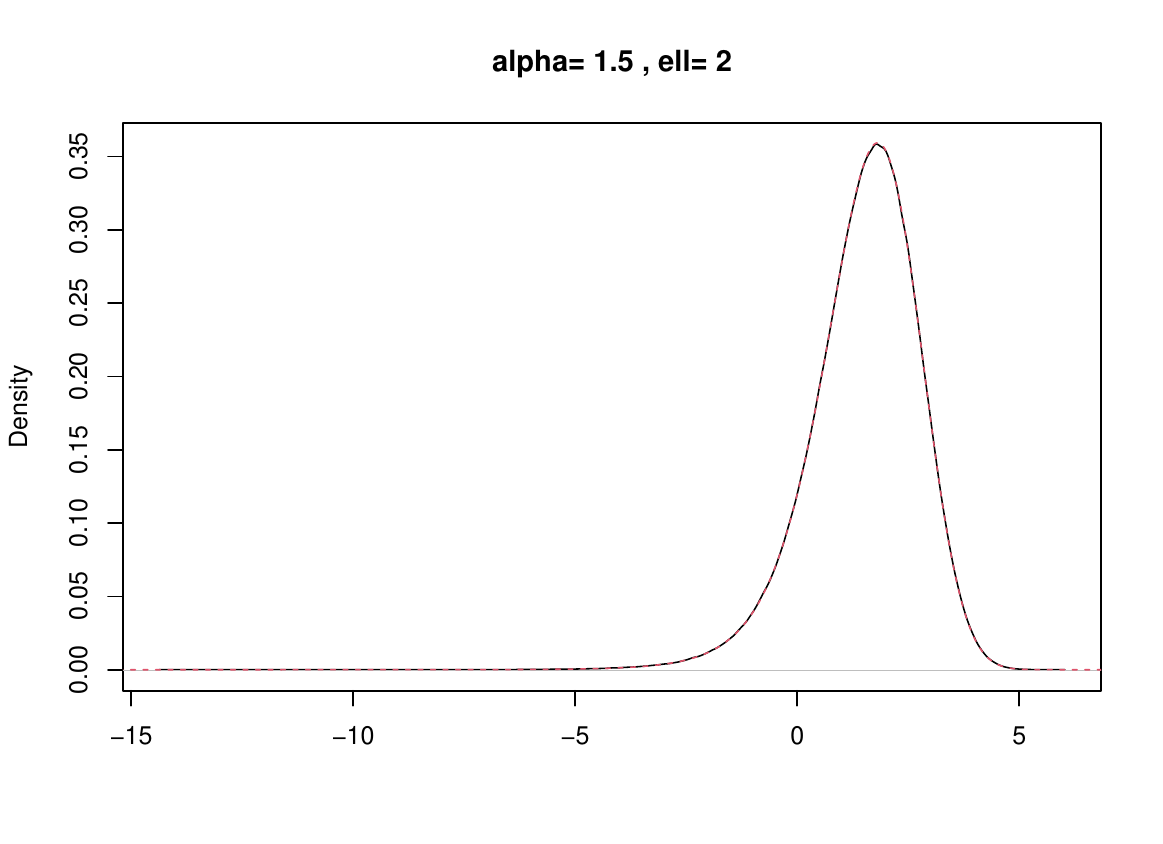} &  \includegraphics[width=0.31\textwidth]{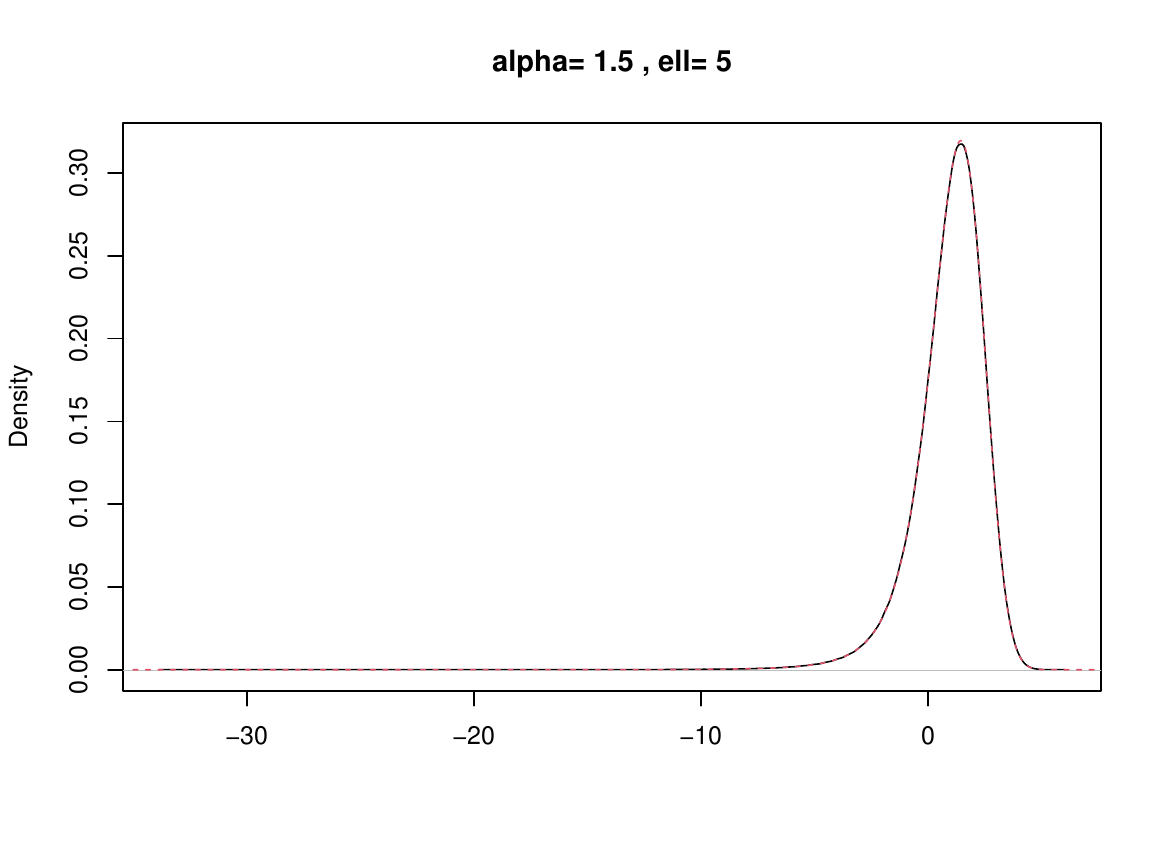}  \\
\includegraphics[width=0.31\textwidth]{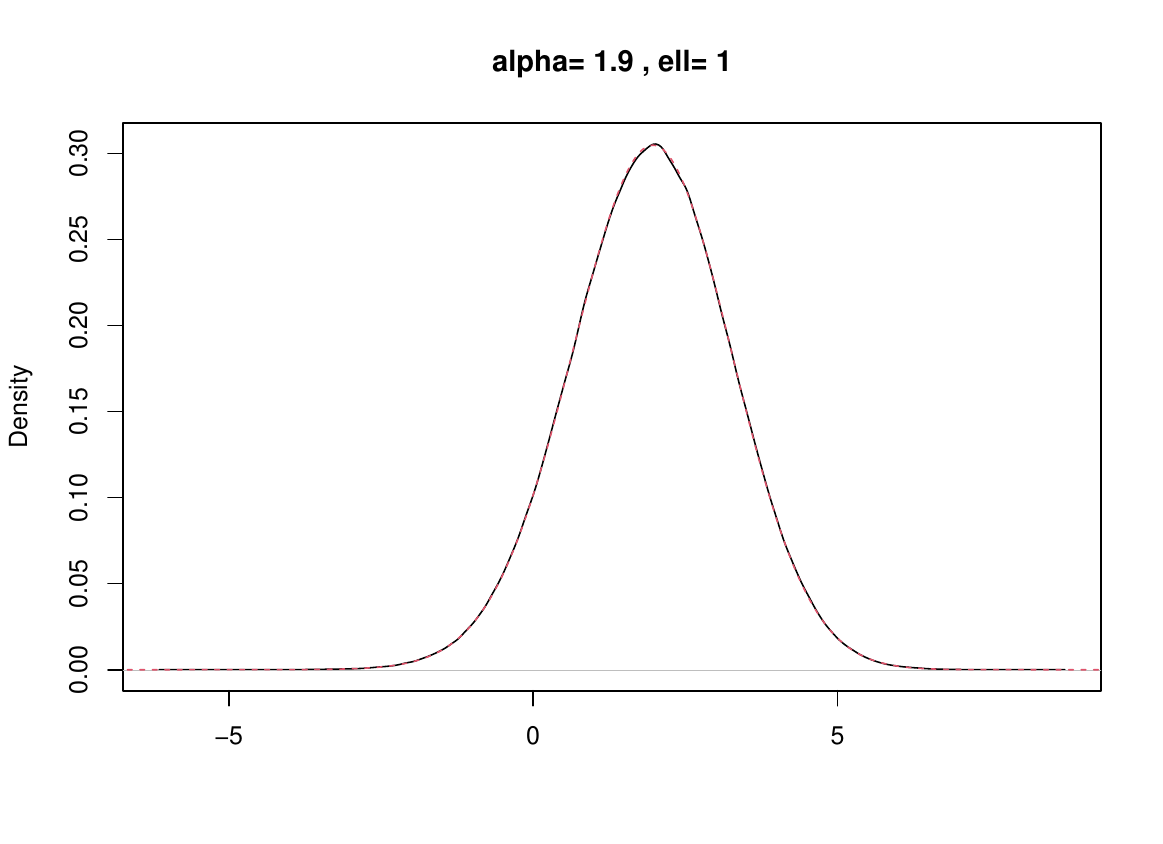} & \includegraphics[width=0.31\textwidth]{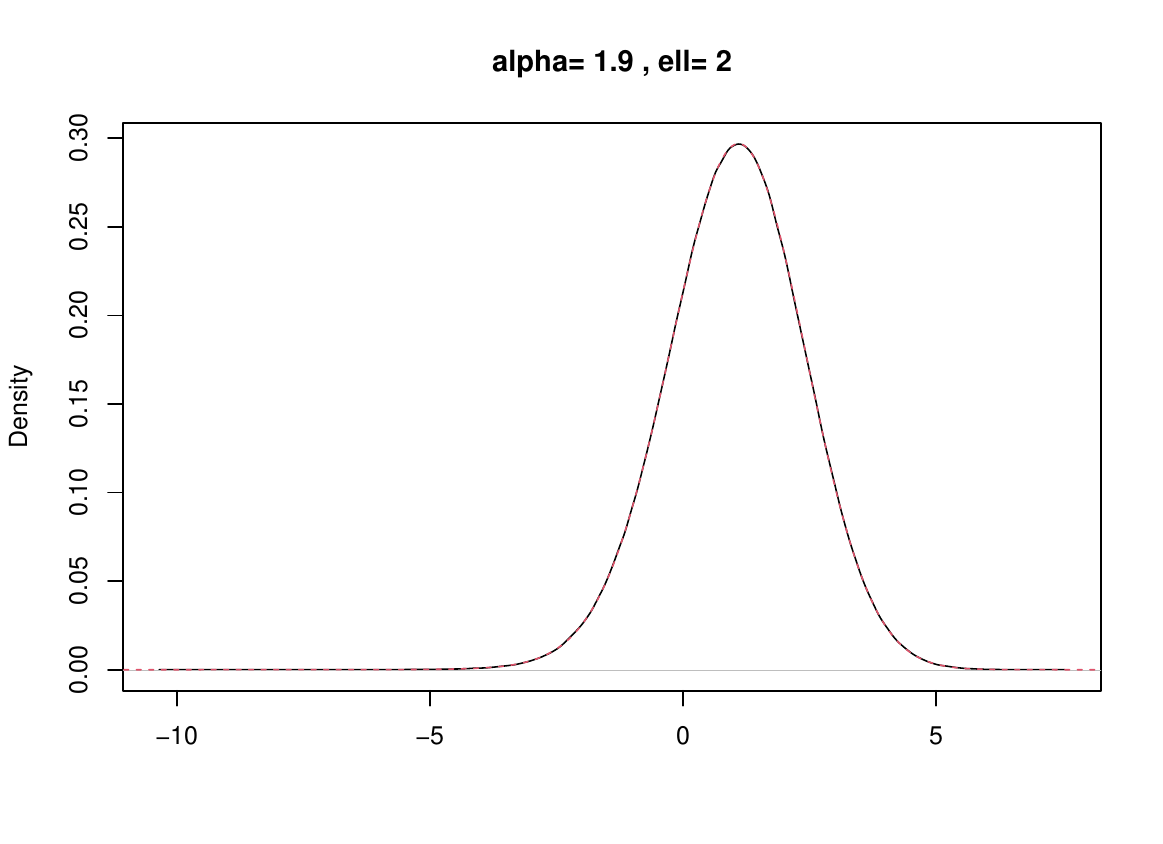} &  \includegraphics[width=0.31\textwidth]{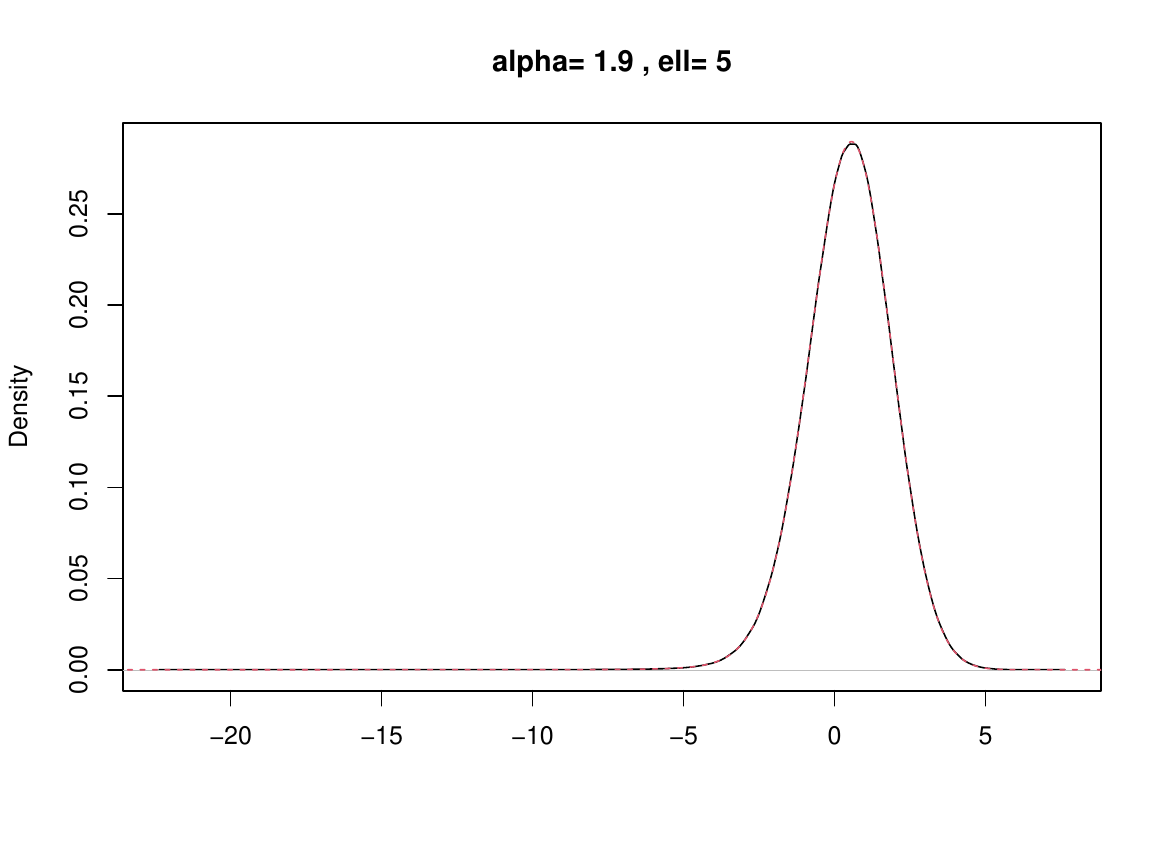} 
\end{tabular}
\caption{For several choices of the parameters, we used Algorithm \ref{alg: TS} to simulate a dataset from the corresponding TS distribution. For each simulated dataset, we plot the 
KDE with the theoretical density overlaid.  In all cases, we take $\sigma=1$.}\label{fig: sim results}
\end{figure}

\section{Refinement}\label{sec: ref}

It can be checked that $C_1$ approaches $\infty$ as $\ell\to0$, with the other parameters fixed. Thus, small values of $\ell$ can lead to large values of $K$. In this section, we introduce a methodology that can help in such situations. A similar methodology for the case $\alpha\in(0,1)$ is discussed in \cite{Hofert:2011} and \cite{Grabchak:2019}. To simplify the discussion, throughout this section we emphasize the dependence of  $K$ on $\ell$ by writing $K(\ell)$. Of course, $K$ also depends on other parameters.

Assume that we want to simulate from $\mathrm{TS}_-(\alpha,\ell,1)$ with $\ell$ small and $K(\ell)$ large and that we cannot improve the situation significantly by better calibrating the tuning parameters. We can consider the following approach. Select an integer $m\ge2$ and simulate $m$ observations $X_1,X_2,\dots,X_m\iid\mathrm{TS}_-\left(\alpha,m^{1/\alpha}\ell,1\right)$. Proposition \ref{prop: get sigma} implies that
\begin{eqnarray}\label{eq: ag}
Y = m^{-1/\alpha}\left(X_1+X_2+\cdots+X_m\right)\sim \mathrm{TS}_-(\alpha,\ell,1).
\end{eqnarray}
When $\alpha=1$, we also need to add $\frac{2}{\pi}\log m$. This gives an alternative approach to simulating from this distribution. On the one hand, this requires simulating $m$ random variables to get just one observation. On the other hand, $K\left(m^{1/\alpha}\ell\right)$ may be much smaller than $K(\ell)$, and the tradeoff may be worth it. Specifically, this is an improvement when $mK\left(m^{1/\alpha}\ell\right)\le K(\ell)$.

We now consider a concrete example, specifically simulation from $\mathrm{TS}_-(1.5,0.3,1)$. We take $p_1=p_2=1/2$ and $\epsilon=0.95$, which are reasonable choices. However, here $K(0.3)=206.8$, which is quite high. If we take $m=2$, we can improve things significantly. In this case, we simulate two observations from $\mathrm{TS}_-\left(1.5,0.3*2^{1/1.5},1\right)$, where $0.3*2^{1/1.5}=0.476$. Using the same tuning parameters, we get $K\left(0.3*2^{1/1.5}\right) = 20.33$ and $2K\left(0.3*2^{1/1.5}\right) = 40.67$, which is a huge improvement. 

We perform a small simulation to verify the performance of this approach. First, we consider direct simulation from $\mathrm{TS}_-(1.5,0.3,1)$. We simulate $N=10^7$ observations from the proposal distributions and, after applying rejection sampling, we obtained $n=48,551$ observations. Next, we consider the approach with $m=2$. We again simulated $N=10^7$ observations from the proposal distribution. After rejection sampling we obtained $491,480$ observations from $\mathrm{TS}_-\left(1.5,0.3*2^{1/1.5},1\right)$ and, after aggregation, we get $245,740$ observations from $\mathrm{TS}_-(1.5,0.3,1)$. In Figure \ref{fig: sim results ag}, we plot the KDE from both datasets with the true density overlaid. We can see that the fit is almost perfect.

\begin{figure}
\centering
\begin{tabular}{cc}
\includegraphics[width=0.31\textwidth]{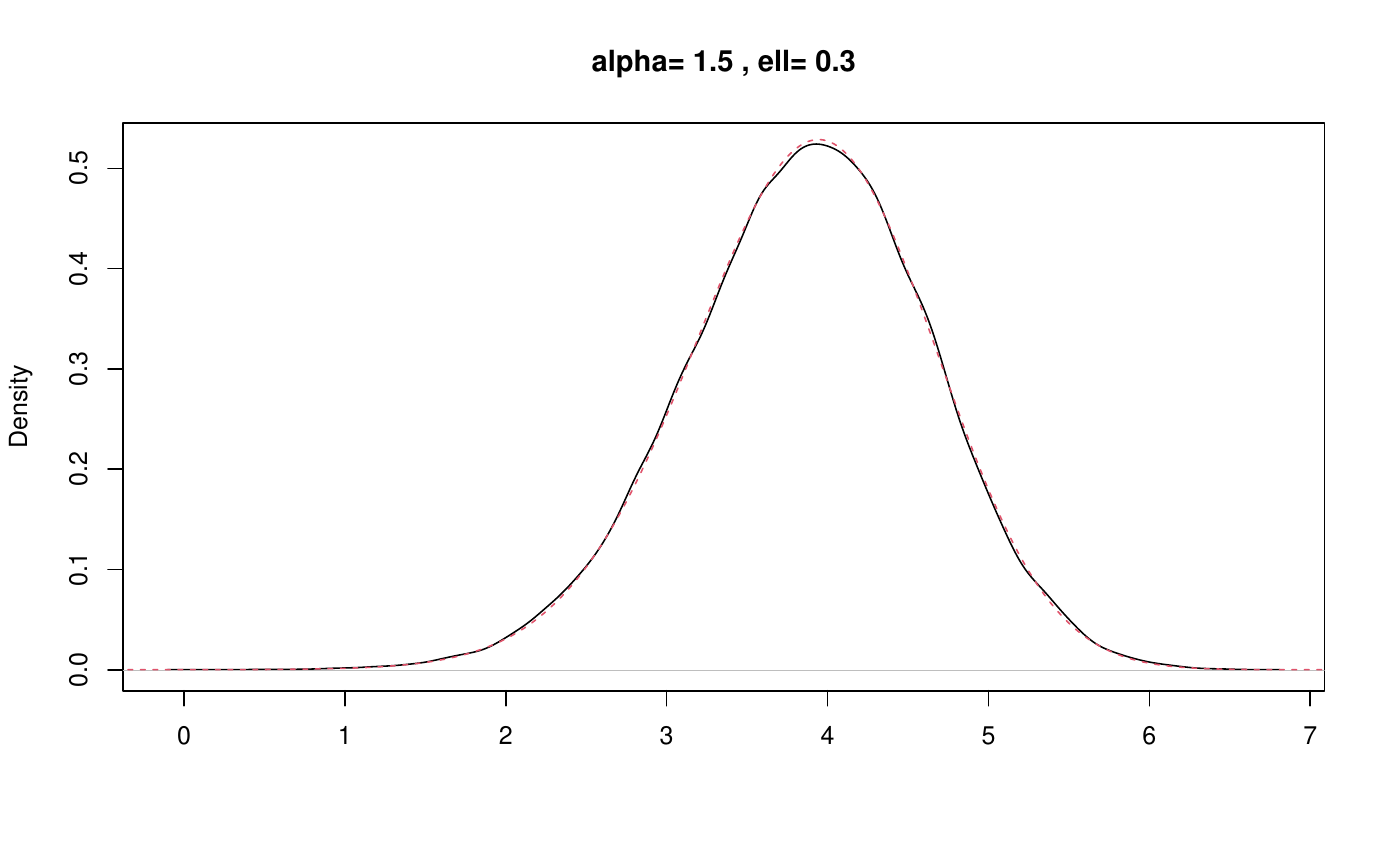} & \includegraphics[width=0.31\textwidth]{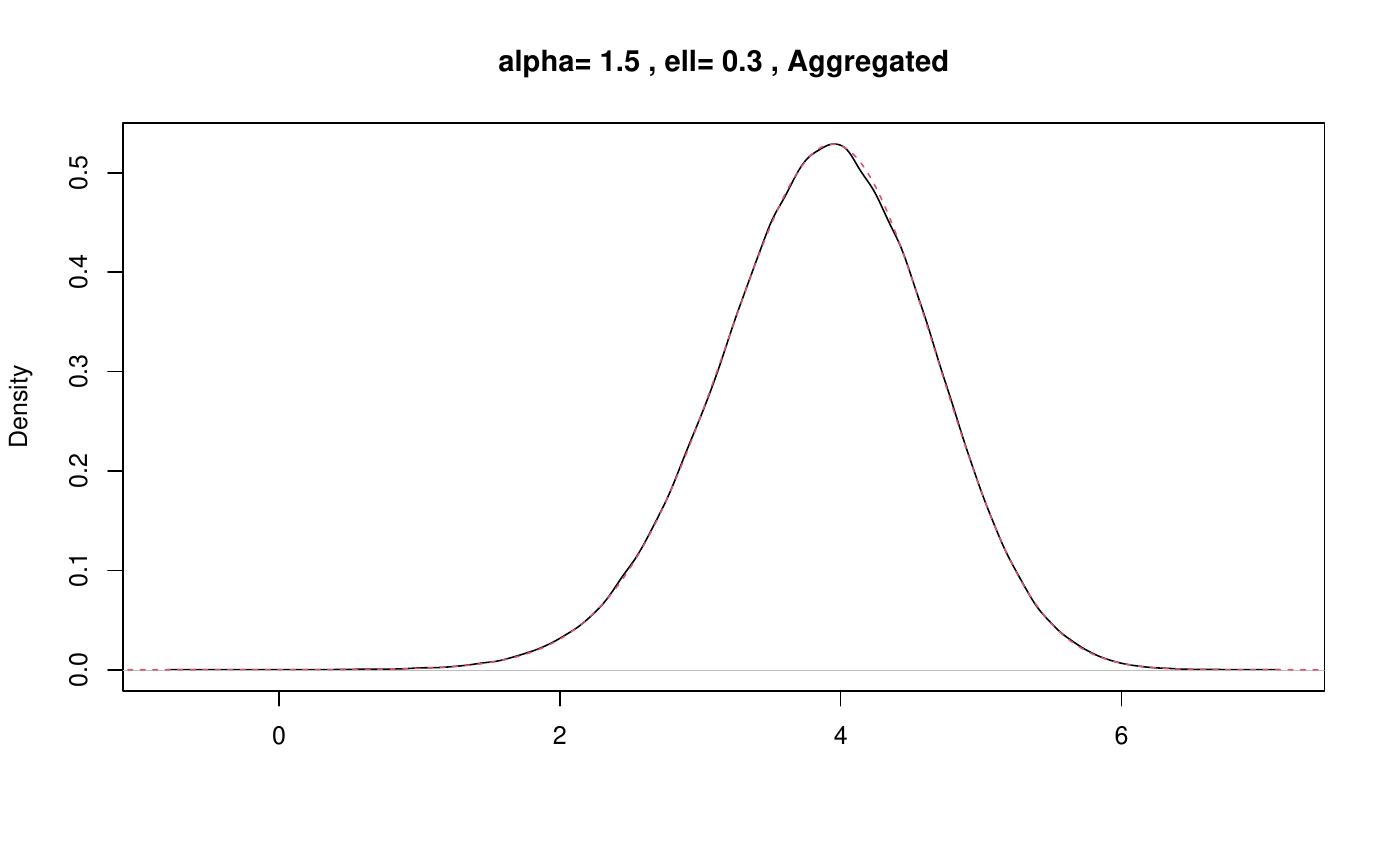} 
\end{tabular}
\caption{We simulated datasets from $\mathrm{TS}_-(1.5,0.3,1)$ using two approaches. The approach on the left uses Algorithm \ref{alg: TS} directly, while the approach on the right uses \eqref{eq: ag} with $m=2$. For each simulated dataset, we plot the KDE with the theoretical density overlaid.}\label{fig: sim results ag}
\end{figure}

\section{Extensions and Conclusions}\label{sec: ext}

In this paper we developed the first exact and computationally tractable method for simulating from fully left-skewed TS distributions in the infinite variation case, where $\alpha\in[0,2)$. Several useful classes of distributions arise as linear combinations of such random variables, and our method extends naturally to these cases. We discuss several such examples in this section and, then, give some directions for future work.

The first example is the class of fully {\em right}-skewed TS distributions, which we briefly mentioned in Remark \ref{remark: right skewed}. These have the same parameters as fully {\em left}-skewed TS distributions and are denote by $\mathrm{TS}_+(\alpha,\ell,\sigma)$. A random variable $X\sim \mathrm{TS}_+(\alpha,\ell,\sigma)$ if and only if $-X\sim \mathrm{TS}_-(\alpha,\ell,\sigma)$.

The next example is the class of classical tempered stable (CTS) distributions. A random variable $X$ is said to have a CTS distribution if
$$
X \eqd X_+-X_- + b,
$$
where $X_+\sim \mathrm{TS}_+(\alpha,\ell_+,\sigma_+)$ and $X_-\sim \mathrm{TS}_+(\alpha,\ell_-,\sigma_-)$ are independent and $b\in\mathbb R$. We denote this distribution $\mathrm{CTS}(\alpha,\ell_+,\sigma_+,\ell_-,\sigma_-,b)$. In the literature, CTS distributions are sometimes also called KoBoL, CGMY, smoothly truncated L\'evy flights (STFL), or just tempered stable distributions.

Multivariate CTS distributions are studied in, e.g., \cite{Xia:Grabchak:2022} and \cite{Xia:Grabchak:2024}. These distributions are characterized by the so-called spectral measure, which is a finite measure on the unit sphere in $d$ dimensions. When the support of the spectral measure is finite, with masses $s_1,s_2,\dots,s_k$, we can simulate from this distribution by taking
$$
X = s_1 X_1 + s_2 X_2+\cdots+s_k X_k,
$$
where $X_1,X_2,\dots,X_k$ are independent random variables with appropriate fully right-skewed TS distributions. Moreover, as shown in \cite{Xia:Grabchak:2022}, even when the support of the spectral measure is infinite, the above leads to a good approximation for an appropriate choice of the masses.

Finally, large and general classes of univariate and multivariate TS distributions are considered in \cite{Rosinski:2007} and \cite{Grabchak:2016}. Theorem 4.18 in \cite{Grabchak:2016} shows that all of these can be approximated by linear combinations of fully right-skewed TS distributions. 

We leave detailed investigations of the application of our methodology to simulation from the above classes of distributions for future work. Another important direction is the development of systematic approaches for calibrating the tuning parameters.

\end{document}